\documentclass[11pt,a4paper]{article}

\usepackage{a4wide}
\usepackage{graphicx}
\usepackage{latexsym}
\usepackage{epsfig}
\usepackage{amsmath} 
\usepackage{amssymb}
\usepackage{amstext}
\usepackage{amsgen}
\usepackage{amsxtra}
\usepackage{amsgen}
\usepackage{amsthm}
\usepackage{color}
\usepackage{enumitem}
\usepackage[algo2e,ruled,noend]{algorithm2e}
\usepackage[backref,colorlinks,linkcolor=red,anchorcolor=green,citecolor=blue]{hyperref}

\usepackage{subfigure}

\def\eps{\varepsilon}
\SetKwInput{kwGiven}{Given}
\SetKwFor{While}{while}{}{end while}%

\def\RR{\mathbb{R}}

\newtheorem{thm}{Theorem}[section]

\newtheorem{lemma}[thm]{Lemma}

\newtheorem{rem}[thm]{Remark}

\theoremstyle{definition}

\newtheorem{definition}[thm]{Definition}

\newcounter{fpcounter}
\renewcommand{\thefpcounter}{\Roman{fpcounter}}

\newenvironment{fp}[2]{%
  \par\vspace*{1em}
\refstepcounter{fpcounter}%
\label{#1}%
\noindent\textbf{Problem~\thefpcounter:}%
}%
{\par\vspace*{1em}}%

\numberwithin{equation}{section}

\newcommand{\intL}{\int_{-L}^{p(\eps)}}
\newcommand{\intT}{\int_\eps^T} 

\begin{document}

\title{Data assimilation in price formation}

\author{Martin Burger\thanks{Department Mathematik, Friedrich-Alexander Universit\"at Erlangen-N\"urnberg, Cauerstrasse 11, 91058 Erlangen. e-mail:martin.burger@fau.de} \and Jan-Frederik Pietschmann\thanks{Technische Universit\"at Chemnitz, Fakult\"at f\"ur Mathematik, Reichenhainer Str. 41, 09126 Chemnitz Germany. jfpietschmann@math.tu-chemnitz.de} \and Marie-Therese Wolfram\thanks{Mathematics Institute, Warwick University, CV4 7AL Coventry, UK and Radon Institute for Computational and Applied Mathematics, Altenbergerstr. 69, 4040 Linz, Austria. e-mail: m.wolfram@warwick.ac.uk}  }

\maketitle

\vspace*{-12pt}

\begin{abstract}
We consider the problem of estimating the density of buyers and vendors in a nonlinear parabolic price formation model using measurements of the price and the transaction rate. 
Our approach is based on a work by Puel et al., see \cite{Puel2002}, and results in a optimal control problem. We analyse this problems and provide stability estimates for the 
controls as well as the unknown density in the presence of measurement errors. Our analytic findings are supported with numerical experiments.
\end{abstract}

\section{Introduction}\label{s:intro}

\noindent In this paper we use techniques developed in the field of data assimilation to predict the dynamics of a nonlinear parabolic free boundary price formation model proposed by Lasry \& Lions in \cite{LL2007}. The Lasry-Lions (LL) model describes the price evolution of a single good traded between a large group of buyers and a large group of vendors. The price
enters as a free boundary, at which trading takes place. After the realisation of a transaction, buyers and vendors immediately sell or rebuy the good at a shifted price. 
The shift in the price is due to the previously paid constant transaction costs. The situation detailed above can be described by the following nonlinear parabolic partial differential equation
\begin{subequations}\label{e:ll}
\begin{align}
&\partial_t f - \frac{\sigma^2}{2}\partial_{xx} f = \Lambda (t) (\delta_{p(t)-a} - \delta_{p(t)+a}),\; x \in \Omega, t > 0,\\
&\Lambda (t) = - \frac{\sigma^2}{2} \partial_x f(p(t),t), ~ f(p(t),t) = 0,\\ 
&f(x,0) = f_0(x),\,p(0) = p_0.
\end{align}
\end{subequations}		
The positive part $f^+ = \max(f,0)$ of the function $f = f(x,t)$ corresponds to the distribution of buyers over the price $x \in \Omega$, the negative part $f^- = \min(f,0)$ to the 
is the vendor distribution over the price. The free boundary $p=p(t)$ corresponds to the price where $f(\cdot,t) = 0$, the function $\Lambda$ to the total number of
transactions executed at that price. The immediate placement and execution of new bids and orders after the trading event are modelled by the Delta Diracs at the shifted
prices $p(t)+a$ and $p(t)-a$, where $a \in \mathbb{R}^+$ denotes the transaction costs. Random changes in the buyer and vendor distribution are included by a Laplacian with constant diffusivity
$\sigma \in \mathbb{R}^+$. We assume that the initial distribution $f_0$ satisfies:
\begin{align}\label{e:compatf0}
 f_0(p_0) = 0,\; f_0(x) > 0 \text{ for }x < p_0\text{ and }f_0(x) < 0\text{ for }x > p_0,\text{ a.e. in }\Omega
\end{align}
and set $\frac{\sigma^2}{2} = 1$. System \eqref{e:ll} can be posed on the positive real line $\Omega = \mathbb{R^+}$ or a bounded interval $\Omega = [0, x_{\max}]$, where
$x_{\max}$ denotes the maximum price. We will consider \eqref{e:ll} on the bounded interval only and impose homogeneous Neumann boundary conditions
\begin{align}\label{eq:neumannbc}
 \partial_x f = 0\text{ on } \partial\Omega
\end{align}
to ensure that the total number of buyers and vendors is constant in time. \\
For convenience, we assume the initial price $p_0$ is normalized to $0$ and only consider its relative change. Hence we work on the shifted domain $[-L,L]$, where $L = \frac{x_{\max}}{2}$. 
Altogether we will consider \eqref{e:ll} with boundary condition \eqref{eq:neumannbc} on $\Omega = [-L,L]$ throughout this manuscript.\\

The LL model \eqref{e:ll} was analysed in a series of papers, cf. \cite{GONZALEZ20113269, MMPW2009, CGGK2009, CMP2011, CMW2011}. Most available results are based on a nonlinear transformation of \eqref{e:ll}, which transforms the problem to the heat equation with nonlinear boundary conditions. This connection provides the main analytical ingredients to study existence and long time behaviour of solutions to \eqref{e:ll}. Lasry and Lions introduced the model on the macroscopic level only, a more detailed microscopic interpretation of the trading process and the respective limit as the number of buyers and vendors tend to infinity was missing. This connection was established by Burger et al., who proved that the
original LL model can be derived from a Boltzmann type model as the number of transactions tends to infinity, see \cite{BCMW2013}. In their approach trading events between buyers and vendors are modelled by ``collisions'', which can also be used to describe price dynamics in case of more general trading rules. The connection
between the Boltzmann-type price formation model and the LL model \eqref{e:ll} was further investigated  in different asymptotic limits in \cite{BCMW2014}.
The LL and Boltzmann-type price formation models are appealing in many respects, especially in terms of analytical tractability. However the resulting price process is deterministic and 
does not give any insights into connections between transactions rates, order flows or price volatility. Markowich et al., \cite{MTW2016} considered a stochastic extensions of 
the original LL model. However this extension did not give realistic price dynamics either. Very recently Cont and M\"uller \cite{CM2019} proposed a stochastic partial differential 
equation with multiplicative noise, which reproduces statistical properties of real price dynamics.\\

In this paper we focus on the inverse problem of determining the buyer-vendor distribution given measurements of the price and the transaction rate on a time interval $[0,T]$. 
This distribution can then be used as an initial value and thus allows us to predict price and transaction rate for $t > T$. More specifically we will investigate the question
\begin{fp}{p:inv}{} 
 Given measurements of the price $p(t)$ and the transaction rate $\Lambda(t)$ in some time interval $[0,T]$, is it possible to predict the price for times $t>T$ ? 
\end{fp}

\noindent Our approach is based on an optimal control approach proposed by J-P. Puel, see \cite{Puel2002,Puel2009}. It is based on a duality argument, which allows to reconstruct the distribution
$f$ at the final time $T$. This is in contrast to standard data assimilation where one tries to recover the initial datum $f_0(x)$. We adopt the strategy of Puel et al. and 
use duality estimates to compute 
linear functionals of $f(T,x)$.  These functionals involve the solution of optimal boundary control problems with PDE constraints. Optimal boundary control problems are well studied in the literature, see e.g. \cite{Lions1992,troltzsch2010optimal,hinze2008optimization}. We will make use 
of an exact null controllability result for parabolic boundary control problems shown in \cite{Chae1996}. Its proof is based on Carleman estimates, a technique commonly used to 
derive exact controllability results (and also uniqueness for inverse problems), see \cite{yamamoto2009carleman,klibanov1992} for details.  A possible numerical realisation of Puel's strategy was presented in 
\cite{Clason2009}.

 Our contributions to the subject of optimal control for parabolic free boundary problems and data assimilation in price formation models are the following:
 \begin{itemize}
 \item We present the first approach to reconstruct the buyer- and vendor distribution from measurements of price and transaction rate (to the author's knowledge).
 \item We generalise the data assimilation approach of Puel et al., see \cite{Puel2002}, to free boundary value problems and evolving domains.
 \item We provide stability estimates, which give novel insights into the influence of measurement errors on the price dynamics.
 \item We propose a computational strategy to implement the developed framework numerically.
\end{itemize}
\noindent This paper is organized as follows: The proposed framework is based on several analytic results, which will be presented in Section \ref{sec:direct}.
 The data assimilation problem itself
is discussed in Section \ref{sec:da}. Section \ref{sec:stability} is devoted to stability in the presence of measurement errors and we conclude by presenting numerical experiments in Section \ref{sec:numerics}.

\section{Preliminary results}\label{sec:direct}

In this section we provide analytic tools and results of the forward problem and define the respective adjoint problem, which will be used in the optimal control approach. \\
The presented results rely on the following assumptions:
\begin{enumerate}[label=(A\arabic*)]
\item  $ f_0(p_0) = 0,\; f_0(x) > 0 \text{ for }x < p_0\text{ and }f_0(x) < 0\text{ for }x > p_0$.\label{a:compatf0}
\item  For every $t \in [0,T]$, there exists a constant $\underline{p}>a$ such that $-L + \underline{p} \le p(t) \le L - \underline{p}$.\label{a:pbounds}
 \end{enumerate}
Assumption \ref{a:compatf0} is the necessary compatibility condition for the initial datum $f_0$ (which we already stated in \eqref{e:compatf0}), while \ref{a:pbounds} ensures that the price stays sufficiently far away from the interval boundaries. Note that the restriction on $p(t)$ is not severe in the context of inverse problems: Since we will assume later on that we know measurements of $p(t)$ in some time interval $[0,T]$, we can always chose the domain size $L$ (within realistic bounds) such that the condition $p(t) \in (-L + a, L - a)$ is satisfied. As $p(t)$ is continuous, we also know it will stay in $(-L+a,L-a)$ for some time so that it is safe to predict for $t>T$.\\

\subsection{Nonlinear transformation of the model}
We start by discussing the nonlinear transformation which converts \eqref{e:ll} to a linear heat equation. This connection was exploited in almost all analytic 
results as well as computational methods. It is based on the fact that the second derivative of the buyer vendor distribution $f$ at the price $p(t)\pm a$ behaves 
like $\Lambda(t) \delta_{p(t)\pm a}$
Thus, shifting the function by multiples of $\pm a$ and adding them up 'eliminates' the singularity on the right hand side.  
More precisely, for $\Omega = \RR$, we define
\begin{align}\label{e:trans}
F(x,t) &= 
\begin{cases}
\phantom{-}\sum_{n=0}^{\infty} f^+(x+na, t),~~x < p(t)\\
-\sum _{n=0}^{\infty} f^-(x-na,t),~~x > p(t).
\end{cases}
\end{align}
Then the function $F = F(x,t)$ satisfies the heat equation
\begin{subequations}\label{eq:heat}
 \begin{align}
   \partial_t F(x,t) - \partial_{xx} F(x,t) &= 0,\quad \quad x \in \RR, t > 0,\\
  F(x,0) &= F_0(x),\; x \in \RR,
 \end{align}
\end{subequations}
with the transformed initial datum
\begin{align*}
F_0 (x) &= 
\begin{cases}
\phantom{-}\sum_{n=0}^{\infty} f_0^+(x+na),~~x < p_0\\
-\sum _{n=0}^{\infty} f_0^-(x-na),~~x > p_0.
\end{cases}
\end{align*}
Since we consider \eqref{e:ll} with homogeneous Neumann boundary conditions on the interval $(-L,L)$, 
the sum in \eqref{e:trans} is finite. If we assume that the initial price is a multiple of $a$, then the transformed 
initial condition is given by
\begin{align}\label{e:trans_bnd}
F_0(x,t) &= 
\begin{cases}
\phantom{-}\sum_{n=0}^{k_L} f^+_0(x+na),~~x < p_0\\
-\sum _{n=0}^{k_R} f^-_0(x-na),~~x > p_0.
\end{cases}
\end{align}
with
\begin{align*}
k_L = (p_0 + L)/a \text{ and } k_R = (L-p_0)/a.
\end{align*}
We recall that the  solution of the original LL model \eqref{e:ll} can be computed by
\begin{align*}
f(x,t) = F(x,t) - F^+(x+a,t) + F^-(x-a,t).
\end{align*}
This back-transformation allows us to deduce the corresponding transformed Neumann boundary conditions
\begin{subequations}\label{eq:bcnonlin}
 \begin{align}
   \partial_x F(-L,t) &= \partial_x F(-L+a,t),\\
  \partial_xF(L,t) &= \partial_x F(L-a,t),
\end{align}
\end{subequations}
\begin{rem} As explained above the non-linear transformation is tailored specifically to \eqref{e:ll}, i.e. the fact that first derivates of solutions to \eqref{e:ll} at $x=p(t)$ are of the form $\pm \lambda \delta_{p(t) \pm a}$ and thus, when summing up shifted solutions the terms on the right hand side vanish. As soon as the structure of the equation is changed, e.g. by adding new non-linear terms or constraints (that act away from $x=p(t)$), the transformation will no longer work.		
\end{rem}

\subsection{Existence and regularity of the price}
In the following we provide additional existence and regularity results for the direct problem. Note that they are not optimal in terms of regularity but sufficient to define all quantities that we shall need in the sequel. 

\begin{thm}[Existence of $f$, $p(t)$] Let $f_0\in L^2(-L,L)$ and $p_0 \in (-L+a, L-a)$ satisfy \ref{a:compatf0}.
Then the BVP \eqref{e:ll} has a global solution conserving the total mass of buyers and vendors iff the
zero level set $p=p(t)$ of the solution of \eqref{eq:heat}--\eqref{eq:bcnonlin} satisfies $p(t) \in (-L + a, L-a)$ for all $t > 0$. Then the free boundary $p(t)$ converges 
to the stationary price $p_\infty \in (-L + a, L-a)$.
\end{thm}
\begin{proof}
The proof is mainly based on the definition of the transformation \eqref{e:trans}, see \cite{CMW2011} for details.
\end{proof}

Note that the stationary price is determined by the initial mass of buyers and vendors as well as the transaction rate $a$. In particular
\begin{align}\label{e:pstat}
 p_{\infty} = \frac{2M^l L - a (M^l-M^r)}{2(M_l + M_r)} - \frac{L}{2}
\end{align}
where $M^l = \int_-L^{p_0} f_0(x) dx$ and $M^r = \int_{p_0}^L f_0(x) dx$.
The presented analysis of the adjoint and assimilation problem relies on the  following regularity result for the price $p=p(t)$.

\begin{lemma}[Regularity of $p(t)$]\label{lem:regp} Let $f_0\in L^2(-L,L)$ and $p_0 \in (-L+a, L-a)$  satisfy  \ref{a:compatf0}. Then $p(t) \in C^1([\eps,T])$ for $\eps > 0$.
 \end{lemma}
 \begin{proof}
   The results is a direct consequence of the fact that $F(x,t)$ is smooth in space and time for all $t>0$ and of the boundedness of $\Lambda$. Indeed, differentiating the 
   relation $F(p(t),t)=0$ yields
  \begin{align}\label{eq:diffp0}
   p'(t) = \frac{\partial_t F(p(t),t)}{\Lambda(t)} = \frac{\partial_{xx} F(p(t),t)}{\Lambda(t)},
  \end{align}
and therefore
\begin{align*}
 \sup_{t\in[\eps,T]} p'(t) \le \frac{\|\partial_{xx} F\|_{C([-L,L]\times [\eps,T])}}{\underline\Lambda},
\end{align*}
where the parabolic version of Hopf's Lemma applied at $x=p(t)$ ensures that $\underline\Lambda = \min_{t \in [\eps,T]} \Lambda(t) > 0$.
 \end{proof}
 \begin{rem}
The regularity of the price $p$ as well as the buyer-vendor density $f$ at the initial time is crucial to define the transformation between the time-dependent domains $[-L,p(t)]$ and $[L,p(t)]$ and the reference domain $[0,1]$ (see Subsection
\ref{s:evolspace}). It is also important for the exact controllability results of Theorem \ref{thm:exact}. Therefore we will work the temporal domain $[\eps,T]$ instead of $[0,T]$ for some fixed $\eps > 0$ in the following only. 
\end{rem}

\subsection{Evolving spaces and the transformation to fixed domains}\label{s:evolspace}

A crucial step in the subsequent analysis is the splitting of the domain $\Omega$ into the part left and right of the price $p(t)$ (illustrated in Figure \ref{f:sketchspacetimedomain}). We introduce the domains
\begin{align*}
 \Omega_\triangleleft = [-L,p(t)] ,\quad \Omega_\triangleright = [p(t),L] ,\text{ and } \Omega = [0,1],
\end{align*}
as well as 
\begin{align*}
 Q_\triangleleft = \Omega_\triangleleft \times [\eps,T],\quad Q_\triangleright = \Omega_\triangleright \times [\eps,T],\text{ and } Q = \Omega \times [\eps,T].
\end{align*}
\begin{figure}
\centering
 \includegraphics[width=0.6\textwidth]{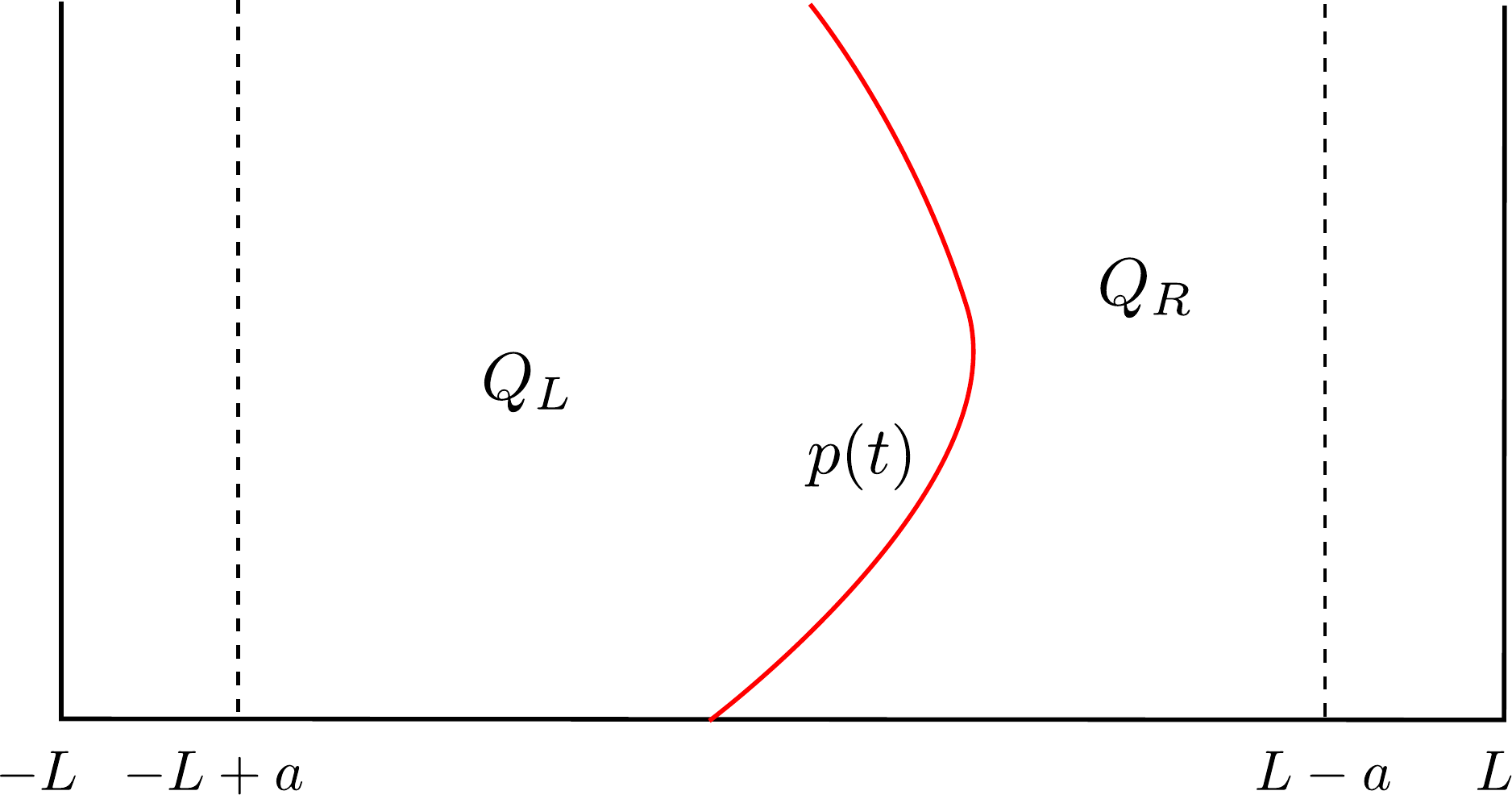}
 \caption{Sketch of the space-time domain and the splitting into $Q_\triangleleft$ (left part) and $Q_\triangleright$ (right part)}
 \label{f:sketchspacetimedomain}
\end{figure} 

Following \cite{Alphonse2015}, we define evolving Bochner spaces on these domains. We present the construction for the left domain $\Omega_\triangleleft=[-L,p(t)]$ only, since the argument for the right
domain is analogous. First denote by $H_\triangleleft^1(t) := H^1((-L,p(t))$ the evolving Hilbert space. Next we define the map
$\phi_t : H_\triangleleft^1(\eps) \to H_\triangleleft^1(t)$ by 
$$
\phi_{t} u(x) = u(\kappa x + \zeta L),
$$
with $\kappa = \frac{p(\eps)+L}{p(t)+L}$ and $\zeta = \frac{p(\eps)-p(t)}{p(t)+L}$ for all $ -L \le x \le p(t)$ and $ \eps \le t \le T$.
The function $\phi_t$ is obviously continuous and reduces to the identity at $t=\eps$. It is also a homeomorphism as its inverse 
$$
\phi_{-t} u(x) = u(\kappa^{-1}x + \zeta^{-1}L) \text{ for } -L \le x \le p(\eps) \text{ and } \eps \le t \le T,
$$
is continuous as well. This allows us to introduce the evolving Bochner spaces (as in \cite[Definition 2.7]{Alphonse2015})
\begin{align} 
L_{H_\triangleleft^1}^{2}&=\{u :[\eps, T] \rightarrow \bigcup H_\triangleleft^1(t) \times\{t\}, t \mapsto(\bar{u}(t), t) \;|\; \phi_{-(\cdot)} \bar{u}(\cdot) \in L^{2}\left(\eps, T ; H_\triangleleft^1(\eps)\right) \}, \\ 
L_{(H_\triangleleft^1)^{*}}^{2} &=\left\{g :[\eps, T] \rightarrow \bigcup (H_\triangleleft^1)^{*}(t) \times\{t\}, t \mapsto(\bar{g}(t), t) \;|\; \phi_{(\cdot)}^{*} \bar{g}(\cdot) \in L^{2}\left(\eps, T ; (H_\triangleleft^1(\eps))^{*}\right)\right\}, 
\end{align}

and, following again \cite{Alphonse2015}, make the identification of $u(t) = (\bar u(t),t)$  with $\bar u(t)$ for $u \in L_{H_\triangleleft^1}^{2}$ (and likewise in $L_{(H_\triangleleft^1)^{*}}^{2}$). 
The space of continuously differentiable functions on evolving Bochner spaces is given by
$$
C_{H_\triangleleft^1}^{k}=\left\{\xi \in L_{H_\triangleleft^1}^{2} | \phi_{-(\cdot)} \xi(\cdot) \in C^{k}\left([\eps, T] ; H_\triangleleft^1(\eps)\right)\right\} \quad \text { for } k \in\{0,1, \ldots\}.
$$
Thus we can, as in \cite[Definition 2.20]{Alphonse2015}, to give a notion of time (material) derivative as
$$
\dot{\xi}(t) :=\phi_{t}\left(\frac{d}{d t}\left(\phi_{-t} \xi(t)\right)\right) \in C_{H_\triangleleft^1}^{0}, 
$$
for any $\xi \in C_{H_\triangleleft^1}^{1}$. Then we can finally define the space used for the notion of weak solutions, namely
\begin{align}\label{eq:evolvingsolutionspace}
W\left(H_\triangleleft^1, (H_\triangleleft^1)^{*}\right)=\left\{u \in L_{H_\triangleleft^1}^{2}\; | \;\dot{u} \in L_{(H_\triangleleft^1)^*}^{2} \right\}.
\end{align}
The definitions of the respective quantities $L_{H_\triangleright^1}^{2}$, $L_{(H_\triangleright^1)^{*}}^{2}$, $C_{H_\triangleright^1}^{k}$, and
$W\left(H_\triangleright^1, (H_\triangleright^1)^{*}\right)$ are analogous.\\
  
While the previous definitions allow us to directly work in a noncylindrical domain, it is sometimes also useful consider the transformation to the 
fixed domain $Q=[0,1]\times [\eps,T]$. Hence we introduce transformations which map $Q_\triangleleft$ and $Q_\triangleright$ to $Q$:
\begin{subequations}\label{eq:trafo}
\begin{alignat}{2}
 T_\triangleleft :& ~Q_\triangleleft \to Q,\qquad   & T_\triangleright :& ~Q_\triangleright \to Q,\\
 (x,t) \mapsto &(\frac{x+L}{L+p(t)},t), \qquad & (x,t) \mapsto& (\frac{-L+x}{-L+p(t)},t).
\end{alignat}
\end{subequations}
Note that due to assumption (A1), $T_\triangleleft$ and $T_\triangleright$ are well-defined and that $T_\triangleright$ actually flips the domain, i.e. it swaps left and right boundary points.

\subsection{Adjoint equations}
The next ingredient will be two adjoint equations, posed on the domains $Q_\triangleleft$ and $Q_\triangleright$, respectively.
\begin{definition}[Adjoint equations]\label{thm:existence_adjoint} For any $\eps> 0$, $\psi_\triangleleft \in L^2(-L,p(T))$, $\psi_\triangleright \in L^2(p(T),L)$, $u_\triangleleft,\,u_\triangleright \in L^2(\eps,T)$ and $T>0$, we introduce the backward in time adjoint equations
 \begin{subequations}\label{eq:adj1}
 \begin{alignat}{2}
 -\partial_t\Phi_\triangleleft(x,t) - \partial_{xx}\Phi_\triangleleft(x,t) &= 0,& \quad &\text{in }  Q_\triangleleft\\ 
  \partial_x \Phi_\triangleleft(-L,t) &= 0,& \quad &\text{for } t \in [T,\eps]\\
  \Phi_\triangleleft(p(t),t) &= u_\triangleleft(t),\;& \quad &\text{for } t \in [T,\eps]\\
  \Phi_\triangleleft(x,T) &= \Psi_\triangleleft(x),\;& \quad &\text{for } x \in \Omega_\triangleleft.
\end{alignat}
\end{subequations}
and
\begin{subequations}\label{eq:adj2}
 \begin{alignat}{2}
 -\partial_t\Phi_\triangleright(x,t) - \partial_{xx}\Phi_\triangleright(x,t) &= 0,& \quad &\text{in } Q_\triangleright \\ 
  \partial_x \Phi_\triangleright(L,t) &= 0,\;& \quad &\text{for } t \in [T,\eps]\\
  \Phi_\triangleright(p(t),t) &= u_\triangleright(t),\;& \quad &\text{for } t \in [T,\eps]\\
  \Phi_\triangleright(x,T) &= \Psi_\triangleright(x),\;& \quad &\text{for } x \in \Omega_\triangleright.
\end{alignat}
\end{subequations}
\end{definition}
Applying the existence theory of, e.g. \cite{Alphonse2015}, for equations on evolving domains, we obtain the following theorem.
\begin{thm}\label{thm:exadj} Let $p \in C^1([\eps,T])$ be given. Then, for every $\Psi_\triangleleft \in L^2(\Omega_\triangleleft)$, $u_\triangleleft\in L^2(\eps,T)$ and every $\Psi_\triangleright \in L^2(\Omega_\triangleright)$, $u_\triangleright\in L^2(\eps,T)$ there exist unique solutions $\Phi_\triangleleft$ and $\Phi_\triangleright$ to \eqref{eq:adj1} and \eqref{eq:adj2}, respectively. Furthermore, we have
\begin{align}
 &\Phi_\triangleleft \in W\left(H_\triangleleft^1, (H_\triangleleft^1)^{*}\right),\\
 &\Phi_\triangleright \in W\left(H_\triangleright^1, (H_\triangleright^1)^{*}\right).
\end{align}
\end{thm}
With the help of the transformations $T_\triangleleft$ and $T_\triangleright$, equation \eqref{eq:adj1}
and \eqref{eq:adj2} can be transformed into a generic problem of the form
\begin{subequations}\label{eq:heatgeneric}
 \begin{alignat}{2}
  -\partial_t \Phi  - a(t)\partial_{yy}\Phi + b(t)y\partial_y \Phi  &= 0,&\quad  &\text{for } (y,t) \text{ in  }Q\\
   \partial_y \Phi(0,t)&=0,\;&\quad &\text{for }t \in [T,\eps]\\
   \Phi(1,t) &= u(t),\;&\quad & \text{for }t \in [T,\eps]\\
   \Phi(y,T) &= \Psi(y),&\quad &y \in (0,1).
 \end{alignat}
\end{subequations}
For \eqref{eq:adj1} we define $(y,t) = T_\triangleleft(x,t)$ and compute
\begin{align}\label{eq:transform_left}
 a(t) = \frac{1}{(p(t)+L)^2},\; b(t) = \frac{p'(t)}{(p(t)+L)},\; u(t) = u_\triangleleft(t)\text{ and }\Psi(y) = \Psi_\triangleleft((p(T)+L)y-L),
\end{align}
while for \eqref{eq:adj2} and $(y,t) = T_\triangleright$ we obtain
\begin{align}
 a(t) = \frac{1}{(p(t)-L)^2},\; b(t) = \frac{p'(t)}{(p(t)-L)},\; u(t) = u_\triangleright(t)\text{ and } \Psi(y) = \Psi_\triangleright((p(T)-L)y+L).
\end{align}

Note that in view of Lemma \ref{lem:regp} and Assumption \ref{a:compatf0}, the coefficients $a$ and $b$ are (in both cases) continuous and uniformly bounded by	
\begin{align}\label{eq:bnd_a_b}
 \frac{1}{(2L-\underline p)^2} < a(t) \le \frac{1}{\underline p^2}\quad \text{ and }\quad 0 \le b(t) \le \frac{\|p\|_{C^1([0,T])}}{\underline p^2},
\end{align}
as there may be points with $p'(t)= 0$.
Thus, standard existence and regularity results for linear diffusion--convection equations on fixed domains,  such as \cite[Theorem 5.2]{Ladyzenskaja1968}, can be used to 
ensure the solvability of \eqref{eq:heatgeneric}.

\section{Data assimilation problem}\label{sec:da}

We now turn to the main part of this paper - the inverse or data assimilation problem \ref{p:inv}.
In classic data assimilation approaches one would use the measurements of $p = p(t)$ and $\Lambda = \Lambda(t)$ on $[0,T]$ to reconstruct the initial datum $f_0(x)$ of \eqref{e:ll}. 
Here we follow an alternative approach proposed by Puel et al., see \cite{Puel2002,Puel2009},  and estimate the buyer-vendor distribution at the final time, that is $f(x,T)$ instead.
This requires the solution of additional optimal control problems, which are, however, well posed if an appropriate regularisation (penalty) is added.

To use Puel's strategy in our setting, we will estimate the densities of buyers and of vendors separately (that is on the right and left of the free boundary).
The reconstruction is based on the following two duality estimates:

\begin{thm} Let $f_0\in L^2(\Omega)$ satisfying assumption \ref{a:compatf0} and let $f\in L^2(0,T;H^1(\Omega))$ be a solution to \eqref{e:ll} with corresponding price, $p \in C^1([\eps,T])$ satisfying (A2). Furthermore, let
\begin{equation} \Phi_\triangleleft \in W\left(H_\triangleleft^1, (H_\triangleleft^1)^{*}\right), \quad
 \Phi_\triangleright \in W\left(H_\triangleright^1, (H_\triangleright^1)^{*}\right) \end{equation}
satisfy \eqref{eq:adj1} and \eqref{eq:adj2}, respectively.
Then, the following duality identities
\begin{subequations}\label{e:duality}
 \begin{align}
 \int_{-L}^{p(T)} f(x,T)\Psi_\triangleleft(x)\; dx &= \int_{-L}^{p(\eps)} f(x,\eps)\Phi_\triangleleft(x,\eps)\; dx + \int_\eps^T \Lambda(t)  (\Phi_\triangleleft(p(t)-a) - u_\triangleleft(t))\;dt\label{eq:orthorel1},\\
 \int_{p(T)}^{L} f(x,T)\Psi_\triangleright(x)\; dx &= \int_{p(\eps)}^{L} f(x,\eps)\Phi_\triangleright(x,\eps)\; dx + \int_\eps^T \Lambda(t) (u_\triangleright(t) - \Phi_\triangleright(p(t)+a))\;dt\label{eq:orthorel2},
 \end{align}
 \end{subequations}
hold for arbitrary functions $u_\triangleleft,\,u_\triangleright \in L^2(0,T)$ and every $\eps> 0$.
\end{thm}
\begin{proof} We prove the first estimate only, since the argument for
  \eqref{eq:orthorel2} is the same. We have
\begin{align*}
 \int_{-L}^{p(T)} &f(x,T)\Psi_\triangleleft(x)\; dx - \int_{-L}^{p(\eps)} f(x,\eps)\Phi_\triangleleft(x,\eps)\; dx =\int_\eps^T  \int_{-L}^{p(t)} \partial_t(f(x,t)\Phi_\triangleleft(x,t))\;dxdt \\
 &=\int_\eps^T\int_{-L}^{p(t)} [\partial_{xx} f(x,t)+\lambda(t)\delta_{p(t)-a}]\Phi_\triangleleft(x,t) - f(x,t)\partial_{xx}\Phi_\triangleleft(x,t)\;dxdt\\
 &=\int_\eps^T \left.(\partial_xf\Phi_\triangleleft)\right|^{x=p(t)}_{x=-L} - \left.(f\partial_x\Phi_\triangleleft)\right|^{x=p(t)}_{x=-L}\;dt + \int_\eps^T \Lambda(t)\Phi_\triangleleft(p(t)-a,t)\;dt \\
 &=\int_\eps^T \Lambda(t)(u_\triangleleft(t)+\Phi_\triangleleft(p(t)-a,t))\;dt.
\end{align*}
where we have used the boundary condition \eqref{eq:neumannbc}, $f(p(t),t) = 0$ and the definition of $\Lambda$.
\end{proof}

Now we will use \eqref{eq:orthorel1}--\eqref{eq:orthorel2} to determine $f(x,T)$. 
Since the choice of $\Psi_\triangleleft$ and $\Psi_\triangleright$  in \eqref{eq:orthorel1} and \eqref{eq:orthorel2} is arbitrary and the last term on the right hand 
side contains only known (i.e. computed or measured) quantities, we could obtain a linear functional of $f(x,T)$. The only unknowns are 
the first terms on the respective right hand sides. But since we are free to choose arbitrary boundary data $u_\triangleleft$ and 
$u_\triangleright$, this leads to the null--controllability problems for \eqref{eq:adj1}--\eqref{eq:adj2}. Indeed, if we can chose 
$u_\triangleleft$ and $u_\triangleright$ such that $\Phi_\triangleleft(x,\eps)=0$ and $\Phi_\triangleright(x,\eps)=0$, the unknown terms in both orthogonality relations drop out
and we can reconstruct $f(x,T)$.


\subsection{Optimal control problem}
To conduct the strategy outlined above, we have to solve the optimal control problems
\begin{align}\label{eq:opt1}
\min_{u_\triangleleft\in L^2(\eps,T)} \frac{1}{2}\int_{-L}^{p(\eps)} \Phi_\triangleleft(x,\eps)^2 dx\quad  &\text{ subject to \eqref{eq:adj1}},\\\label{eq:opt2}
\min_{u_\triangleright\in L^2(\eps,T)} \frac{1}{2}\int_{p(\eps)}^L \Phi_\triangleright(x,\eps)^2 dx\quad  &\text{ subject to \eqref{eq:adj2}}.
\end{align}
Since the structure of both problems is the same, we will only discuss the first one. To increase readability, we will drop the subscript $\triangleleft$ and
write $u$, $\phi$, $\ldots$ instead of $u_\triangleleft$, $\phi_\triangleleft$ from now on. The next result states that the optimal control problem is indeed exactly 
null-controllable in the sense of the following definition.
\begin{definition} We say that problem \eqref{eq:opt1} is 
\emph{exactly null-controllable}, if for every initial datum $\Psi \in L^2(\Omega_\triangleleft)$ to \eqref{eq:adj1}  there exists $\bar u\in L^2(\eps,T)$ such that the solution $\Phi$ to \eqref{eq:opt1} with control $u=\bar u$ satisfies
 $\Phi(x,\eps) = 0.$
\end{definition}
The following exact boundary controllability result is based on \cite[Theorem 2.3]{Chae1996}, slightly extended and adapted to our situation. The theorem reads as follows.
\begin{thm}[Exact null-controllability]\label{thm:exact} For every $\Psi \in L^2(\Omega_\triangleleft)$, there exists at least one control $u\in L^2(\eps,T)$ such that 
the solutions $\Phi$ of \eqref{eq:adj1} satisfies
$\Phi(x,\eps) = 0$  { on }$\Omega_\triangleleft$.
Furthermore, there exists a constant $C$ which depends on $p(t)$, $L$ and $T$ such that 
\begin{align}\label{eq:contctrl}
 \|\bar u\|_{L^2(\eps,T)} \le C\|\Psi\|_{L^2(\Omega_\triangleleft)}
\end{align}
holds with $\bar u$ being the control of minimum $L^2$--Norm.
\end{thm}
 \begin{proof} The regularity of the price $p$ allows us to transform the
   problem to a fixed domain using $T_\triangleleft$ defined in \eqref{eq:trafo}. Hence we
   only consider equations of type \eqref{eq:heatgeneric}. 
First we observe that for any positive $\delta > 0$, any solution $\Phi$ to \eqref{eq:heatgeneric} with initial datum $\Psi$ is, by standard parabolic regularity \cite[Chapter 7.1]{Evans2010}, in $L^2(\eps+\delta,T);H^1(0,1))$ with the estimate
\begin{align}\label{eq:contW}
 \|\Phi \|_{L^2(\eps+\delta,H^1(0,1))} \le C \|\Psi\|_{L^2(0,1)}.
\end{align}
Thus, we can assume that already $\Psi \in H^1(0,1)$ holds. Since by lemma \ref{lem:regp}, $p\in C^1([\eps,T])$ (and thus the coefficients $a$ and $b$ in \eqref{eq:heatgeneric} are continuous) we can apply \cite[Theorem 2.3]{Chae1996} to conclude the requested boundary controllability. The continuity estimate \eqref{eq:contctrl} then follows by combining \eqref{eq:contW}, the respective estimate from \cite[Theorem 2.1]{Chae1996} for the distributed control problem and a standard trace inequality.\\
\end{proof}
In order to be able to numerically solve the optimal control problem, we introduce the following regularized version
\begin{align}\label{eq:minreg1}
&\min_{u\in L^2(\eps,T)} \frac{1}{2}\int_{-L}^{p(\eps)} \Phi(x,\eps)^2 dx
  + \frac{\alpha}{2} \int_\eps^T u(t)^2 \; dt \text{ subject to
    \eqref{eq:adj1}}
\end{align}
Standard arguments guarantee the existence of a unique minimizer, see e.g. \cite[Section 3.5]{troltzsch2010optimal}. Calculating the derivatives of 
the corresponding Lagrange functional
\begin{align}\label{eq:lagrange}
  \begin{split}
\mathcal{L} &= \frac{1}{2}\intL \Phi(x,\eps)^2 dx  + \frac{\alpha}{2}\intT u(t)^2 \;dt \\
& + \intT \int_{-L}^{p(t)} G(x,t) \left[-\partial_t \Phi(x,t) -
  \partial_{xx}\Phi(x,t)\right]\; dx dt,
\end{split}
\end{align}
we obtain the first order optimality system
\begin{subequations}\label{eq:heat_opt}
 \begin{alignat}{2}
   \partial_t G(x,t) - \partial_{xx} G(x,t) &= 0,\quad &&\textrm{in } Q_\triangleleft \\
   \partial_x G(-L,t) &= 0, &&\textrm{for }t>\eps \\
    G(p(t),t) &= 0, & &\textrm{for } t > \eps\\
  G(x,\eps) &= -\Phi(x,\eps),\quad &&\textrm{in } \Omega,
 \end{alignat}
\end{subequations}
where $\Phi$ satisfies the adjoint equation \eqref{eq:adj1} and the coupling
 \begin{align}\label{eq:coupling_opt}
\alpha u(t) + \partial_x G(p(t),t) &= 0,\;\textrm{ for } t > \eps.
 \end{align}
The following results examine the convergence of $u$ as $\alpha \to 0$. The proofs are using the same techniques as in \cite{Puel2009}, yet adapted to our boundary control problem.
\begin{thm} For every $\alpha > 0$, denote by $(u_\alpha,\phi_\alpha)$ the corresponding solution to \eqref{eq:minreg1}. Then we have 
\begin{alignat}{2}
 u_\alpha &\to \bar u\text{ in } L^2(\eps,T)\quad &\text{ as }&\alpha\to 0,\\
 \Phi_\alpha &\to \bar \Phi\text{ in }C([\eps,T];H^1(\Omega_\triangleleft))\quad &\text{ as } &\alpha \to 0,
\end{alignat}
where $\bar u$ is the solution to the optimal control problem \eqref{eq:opt1} having minimal $L^2$-norm and $\Phi_\alpha$ and $\bar \Phi$ are the solutions to \eqref{eq:adj1} with boundary data $u_\alpha$ and $\bar u$, respectively.
\end{thm}
\begin{proof} By Theorem \ref{thm:exact}, we know that there exists at least one function solving the exact null controllability problem. Thus, the set of all these controls in $L^2(\eps,T)$ is nonempty. As it is also convex and closed, there exists a unique $\bar u$ having minimal $L^2$-norm. Since $u_\alpha$ minimizes the functional \eqref{eq:minreg1} among all function in $L^2(\eps,T)$ we have
 \begin{align}\label{eq:estalpha}
   \frac{1}{2}\int_{-L}^{p(\eps)} \Phi_\alpha(x,\eps)^2 dx  + \frac{\alpha}{2} \int_\eps^T u_\alpha(t)^2 \; dt \le \frac{\alpha}{2} \int_\eps^T \bar u(t)^2\;dt
 \end{align}
 which implies the (uniform in $\alpha$) bound
 \begin{align}\label{eq:uniforml2_control}
 \frac{1}{2} \int_\eps^T u_\alpha(t)^2 \; dt &\le C.
 \end{align}
Thus, we can extract a subsequence, again labelled $u_\alpha$ that converges weakly to some $\tilde u$ in $L^2(\eps,T)$. Using the weak formulation of \eqref{eq:adj1} and an Aubin-Lions argument, we see that this is sufficient to obtain the convergence 
\begin{align*}
 \Phi_\alpha \to \tilde\Phi \text{ in } C([\eps,T];H^1(-L,p(t))),\quad \text{ as } \alpha \to 0
\end{align*}
and \eqref{eq:estalpha} implies
$ \tilde \Phi(\eps,x) = 0.$
Thus, arguing as in the proof of \cite[Theorem 2.12]{Puel2009}, we can use the fact that $\bar u$ has minimal norm as well as the lower semi-continuity of the norm w.r.t weak convergence to obtain that $\tilde u = \bar u$. This argument also implies norm convergence and the uniqueness of the limit then finally yields
\begin{align*}
 u_\alpha \to \bar u\text{ in } L^2(\eps,T).
 \end{align*}
This also implies $\tilde \Phi = \Phi$ which completes the proof.
\end{proof}
\begin{rem} Understanding the optimal control problem \eqref{eq:opt1} (or \eqref{eq:opt2}) as Tikhonov regularisation, one could ask for convergence rates of $u_\alpha$ to $\bar u$ as $\alpha \to \infty$. Indeed, such rates could be expected under appropriate source conditions on $\bar u$. The interesting point now is to understand the influence of $p(t)$ in the definition of the forward operator in the characterisation of such conditions and also how perturbation in $p$ would influence them. We leave this question for future research.
\end{rem}
%
\section{Stability in the presence of measurement errors}\label{sec:stability}
Assume we have measurements of two different prices $p_1$ and $p_2$ as well as two different transaction rates $\Lambda_1(t)$ and $\Lambda_2(t)$. Can we control the difference in the reconstructions $f_1(x,T)$ and $f_2(x,T)$ as well as the future predicted prices $p_1(t)$ and $p_2(t)$ for $t>T$ in terms of these differences? In this section we will give a positive answer to this question based on the following strategy
\begin{enumerate}
 \item Estimate the error in the optimal controls $u_1$ and $u_2$ in terms of the error in $p_1$ and $p_2$ (Lemma \ref{lem:stab_u}).
 \item Estimate the error in the respective reconstructions $f_1(x,T)$ and $f_2(x,T)$ in terms of errors in price and transaction rate (Lemma \ref{lem:stab_initial}).
 \item Use these results to predict errors in the future price (Lemma \ref{lem:stab3}).
\end{enumerate}
Note however that for the last point we need to make additional regularity assumptions on the reconstructed final data that do not directly follow from our analysis (see Remark \ref{rem:regularity} for details).
We start by assuming 
\begin{enumerate}[label=(A\arabic*),start=3]
\item  $\|p_1 - p_2 \|_{C^1([\eps,T])} \le \delta_p \quad \text{and} \quad \|\Lambda_1 - \Lambda_2\|_{L^2((\eps,T))} \le \delta_\lambda.$ \label{a:plbounds}
\end{enumerate}
W.l.o.g. we only consider the optimality system related to \eqref{eq:opt1},
i.e. the left part  $\Omega_\triangleleft = [-L,p(t)]$ and again drop the subscript $\triangleleft$. Moreover, we transform all equations to the unit interval $[0,1]$, so that the optimality system reads as
\begin{subequations}
  \begin{alignat}{2}\label{eq:transformed_primal}
  -\partial_t \Phi  - a(t)\partial_{yy}\Phi + b(t)y\partial_y \Phi  &= 0, \quad &&\, \textrm{in } Q \\
   \partial_y \Phi(0,t)&=0, \quad&&\text{for } t > \eps\\
   \Phi(1,t) &= u(t), \quad&&\text{for } t > \eps\\
   \Phi(y,T) &= \Psi(y),\quad&&\textrm{for all } 0 \le y \le 1.
  \end{alignat}
 \begin{alignat}{2}\label{eq:transformed_adjoint}
  \partial_t G  - a(t)\partial_{yy}G - b(t)y\partial_y G &= 0, \quad &&\,\textrm{in } Q \\ 
   \partial_y G(0,t)&=0, &&\text{for } t>\eps\\
   G(1,t) &= 0, &&\text{for } t>\eps\\
   G(y,\eps) &= - \Phi(y,\eps),\quad&&\text{for all } 0 \le y \le 1.
 \end{alignat}	
 and the coupling condition
 \begin{align}\label{eq:transformed_coupling}
  \alpha u(t) + \frac{1}{(p(t)+L)}\partial_y G(1,t) &= 0,\; \text{ for }t > \eps,
 \end{align}
\end{subequations}
with $a(t)$ and $b(t)$ as defined in \eqref{eq:transform_left}. Note that the transformed primal and dual equations are still adjoint to one another, yet now with respect to the scalar product
\begin{align}\label{eq:transformed_scalar}
 (u,v) := \int_\eps^T \int_0^1 (p(t) + L) uv\;dxdt.
\end{align}

\begin{lemma}\label{lem:opti_apriori}
Let $\Phi$ and $G$ be the solutions to \eqref{eq:transformed_primal} and \eqref{eq:transformed_adjoint}, respectively. Then we have
\begin{align*}
 \|\Phi\|_{L^\infty((T,\eps);L^2(0,1))} + \|\sqrt{a}\Phi\|_{L^2((T,\eps);H^1(0,1))} &\le C_1\left( 1 + \|\Psi\|_{L^2((0,1))}\right),\\
 \|G \|_{L^\infty((\eps,T);L^2(0,1))} + \|\sqrt{a}G\|_{L^2((\eps,T);H^1(0,1))} &\le C_2\|\Phi(\cdot,\eps)\|_{L^2((0,1))},
\end{align*}
with $C_1 = C_1(\alpha, \underline p, L, T)$ and $C_2 = C_2(\underline p, L, T)$.
\end{lemma}
\begin{proof}
These are standard estimates obtained choosing $\Phi$ and $G$ as test functions in the weak formulation of \eqref{eq:transformed_primal} and \eqref{eq:transformed_adjoint}, respectively. For the first estimate, we additionally used the $L^2$-bound \eqref{eq:uniforml2_control} on the boundary control, which introduced the $\alpha$-dependence in $C_1$.
\end{proof}
Now we are able to prove stability of the optimal control problem in terms of measurement errors in the price.
\begin{lemma}[Stability of $u$]\label{lem:stab_u} Consider two different prices $p_1(t)$ and $p_2(t)$ such that $p_1(\eps) = p_2(\eps)$ and $\|p_1-p_2\|_{C^1([\eps,T])} \le \delta_p$. Denote by $\Phi_1$ and $\Phi_2$ and $G_1$ and $G_2$ the solutions to \eqref{eq:transformed_primal} and \eqref{eq:transformed_adjoint} with $p=p_1$ and $p=p_2$, respectively. Then the following 
stability estimate for the controls $u_1$ and $u_2$ holds:
\begin{align*}
  &\int_0^1 (\Phi_1(x,\eps)-\Phi_2(x,\eps))^2\;dx + \frac{\alpha}{2}\int_\eps^T  (u_1(t)-u_2(t))^2\;dt \le C_3(\alpha, \underline p, L, T, \Psi) \delta_p^2.
\end{align*}
\end{lemma}
\begin{proof}
For each $p_i$ (and corresponding $a_i$, $b_i$), we denote by $G_i$, $\Phi_i$ and $u_i$ the corresponding solutions to the optimality system \eqref{eq:transformed_primal}--\eqref{eq:transformed_coupling} and furthermore
$$ \bar \Phi = \Phi_1 - \Phi_2, \quad 
 \bar G = G_1 - G_2.$$
Then, $\bar \Phi$ and $\bar G$ satisfy, in the weak sense, the equations
\begin{subequations}
 \begin{alignat}{2}\label{eq:transformed_primal_diff}
  -\partial_t \bar\Phi  - a_1(t)\partial_{yy}\bar\Phi + b_1(t)y\partial_y\bar \Phi  &= -(a_1-a_2)\partial_{yy}\Phi_2 + (b_1-b_2)y\partial_y \Phi_2 ,  &\,&\text{in } Q\\ 
   \partial_y \bar \Phi(0,t)&=0, & &\text{for } t > \eps\\
   \bar\Phi(1,t) &= u_1(t)-u_2(t),& &\text{for } t >\eps\\
   \bar \Phi(y,T) &= 0,&& 0 \le y \le 1.
\end{alignat}
\text{and}
 \begin{alignat}{2}	\label{eq:transformed_adjoint_diff}
  \partial_t \bar G  - a_1(t)\partial_{yy}\bar G - b_1(t)y\partial_y \bar G &= -(a_1-a_2)\partial_{yy}G_2 - (b_1-b_2)y\partial_y G_2 ,  &\,&\text{ in } Q\\ 
   \partial_y G(0,t)&=0,&&\text{for } t>\eps\\
   G(1,t) &= 0,&&\text{for } t>\eps\\
   G(y,\eps) &= - \bar\Phi(y),&& 0 \le y \le 1.
 \end{alignat}
\end{subequations}
Note that the following calculations are formal since for now we only know existence of weak solutions and therefore some of the integrals are not defined. In the end we arrive, however, at an estimate which is again well defined and could can be obtained rigorously by directly working with weak solutions. We chose this way of presentation as we believe it to be easier to follow. Thus (formally) taking equation \eqref{eq:transformed_primal_diff} and testing it with $\bar G$ (with respect to the scalar product \eqref{eq:transformed_scalar}) yields
\begin{align*}
   \int_Q (p(t)+L)&\bar G[-\partial_t \bar\Phi  -
    a_1(t)\partial_{yy}\bar\Phi + b_1(t)y\partial_y\bar \Phi]\;dxdt\\
  &= \int_Q (p(t)+L)\bar G[-(a_1-a_2)\partial_{yy}\Phi_2 + (b_1-b_2)y\partial_y \Phi_2]\;dxdt 
\end{align*}
Integrating by parts on the left hand side, using \eqref{eq:transformed_adjoint_diff} and the boundary conditions results in
\begin{align*}
(p(t)&+L)	\int_0^1 \bar \Phi(x,\eps)^2\;dx + \alpha \int_\eps^T a_1(t)(u_1(t)-u_2(t))^2 (p(t)+L)\;dt \\
 &= \int_Q (p(t)+L)\left\{ [-(a_1-a_2)[\partial_{yy}\Phi_2\bar G - \partial_{yy}G_2 \bar\Phi] + (b_1-b_2)y[\partial_y \Phi_2 \bar G - \partial_y G_2 \bar \Phi] \right\}\;dxdt.
\end{align*} 
A final integration by parts to remove the second derivatives on the right hand side gives
\begin{align*}
(p&(t)+L)\int_0^1 \bar \Phi(x,\eps)^2\;dx + \alpha \int_\eps^T a_1(t)(u_1(t)-u_2(t))^2 (p(t)+L)\;dt \\
 &= \int_Q (p(t)+L)\left\{ [-(a_1-a_2)[-\partial_{y}\Phi_2\partial_y G_1 + \partial_{y}G_2 \partial_y\Phi_1] + (b_1-b_2)y[\partial_y \Phi_2 \bar G - \partial_y G_2 \bar \Phi] \right\}\;dxdt\\
 &\quad + \alpha \int_\eps^T (a_1-a_2)(u_1-u_2)u_2(p(t)+L)\;dt.
\end{align*} 
Using the estimates of Lemma \ref{lem:opti_apriori}, the boundedness of $u$ in $L^2$ (see \eqref{eq:uniforml2_control}) and Cauchy's inequality applied to the last term on the right hand side, we have
\begin{align*}
&\int_0^1 \bar \Phi(x,\eps)^2\;dx + \frac{\alpha}{2} \int_\eps^T (u_1(t)-u_2(t))^2\;dt\le  C_4(\underline{p}, L,\Psi,\alpha)(\|a_1-a_2\|^2_{L^\infty(\eps,T)} + \|b_1-b_2\|^2_{L^\infty(\eps,T)}),
\end{align*} 
where we also used the lower bounds \eqref{eq:bnd_a_b} on $a$ and Assumption \ref{a:plbounds} to estimate the expression $(p(t)+L)$ from below by $\underline p$ and above by $L-\underline p$. Using again \eqref{eq:bnd_a_b} yields
\begin{align*}
 \|a_1 - a_2\|_{L^\infty((\eps,T))} \le C_4(\underline p, L)\delta_p,\textrm{ and }\|b_1 - b_2\|_{L^\infty((\eps,T))} \le C_5(\underline p, L)\delta_p.
\end{align*}
Combining this with the previous estimate yields the assertion.
\end{proof}
For the second step of our strategy, we return to the orthogonality relation \eqref{eq:orthorel1} which, transformed to $[0,1]$, reads as
\begin{align}\label{eq:duality_transformed}
\begin{split}
 (p(T)+L)\int_{0}^{1} f(x,T)\Psi(x)\; dx &= (p(\eps)+L)\int_{0}^{1} f(x,\eps)\Phi(x,\eps)\; dx \\
 &+ \int_\eps^T \Lambda(t)  (\Phi(p(t)-a) - u(t))\;dt.
 \end{split}
 \end{align}
 In the presence of errors in $p$ and $\Lambda$ we obtain two different relations and the following stability result.
Note that the above results on the adjoint equations imply solvability for $\Phi$ with continuous dependence on the initial value for any $\Psi \in L^2([0,1])$. Hence, the duality relation uniquely defines $f(\cdot,T) \in L^2([0,1])$ when given $f(\cdot,\epsilon) \in L^2([0,1])$. There is further stable dependence of $f(\cdot,T)$ on the errors in the price and transaction rates, which we make precise by the following result:
 \begin{lemma}[Stability of $f(x,T)$]\label{lem:stab_initial} Let $p_1,\,p_2$ and $\Lambda_1,\,\Lambda_2$ be given functions which satisfy Assumption \ref{a:plbounds} and
 denote by $f_1(x,T)$ and $f_2(x,T)$ the corresponding reconstructed prices calculated using \eqref{eq:duality_transformed}. Then we have
\begin{align*}
\int_{0}^{1} (f_1(x,T)-f_2(x,T))^2\; dx &\le C_6(\underbar p, f(\cdot,\eps), \Psi, \alpha,L,T)( \delta_p + \delta_\lambda).
 \end{align*}
 \end{lemma}
 \begin{proof}
 Subtracting \eqref{eq:duality_transformed} for $(p_1,\lambda_1,u_1)$ and $(p_2,\lambda_2,u_2)$ yields
 \begin{align*}
  &(p_1(T)+L)\int_0^1 (f_1(x,T)-f_2(x,T))\Psi(x)\;dx 	=\\
  &=\underbrace{(p_1(\eps)+L) \int_0^1 f(x,\eps)(\Phi_1(x,\eps)-\Phi_2(x,\eps))\;dx + (p_1(\eps)-p_2(\eps)) \int_0^1 f(x,\eps)\Phi_2(x,\eps)\;dt}_{=:(I)}\\
  &\quad + \underbrace{\int_\eps^T \Lambda_1(t) [ \Phi_1(p_1(t)-a) - u_1(t)) - \Phi_2(p_2(t)-a) - u_2(t))]\;dt}_{=:(II)}\\
  &\quad +\underbrace{\int_\eps^T [\Lambda_1(t) - \Lambda_2(t)] (\Phi_2(p_1(t)-a) - u_2(t))\;dt}_{ =:(III)}- \underbrace{(p_1(T)-p_2(T))\int_0^1 f_2(x,T)\Psi(x)\;dx}_{=:(IV)}
 \end{align*}
We estimate each term of the right hand side separately 
\begin{align*}
 (I) &\le  (L-\underline p) \|f(\cdot,\eps)\|_{L^2((0,1))} \|\Phi_1(\cdot,\eps)-\Phi_2(\cdot,\eps)\|_{L^2(0,1)} \\
 & \quad + \|p_1 - p_1\|_{L^\infty(\eps,T)} \|f(\cdot,\eps)\|_{L^2((0,1))} \|\Phi_2(\cdot,\eps)\|_{L^2((0,1))}\\
 &\le C_7(\underbar p, f, \Psi, \alpha,L,T) \delta_p,
 \end{align*}
where we used Lemmata \ref{lem:opti_apriori} and \ref{lem:stab_u}.
Next we have
\begin{align*}
 (II)& \le \|\lambda_1\|_{L^2(\eps,T)} \|u_1-u_2\|_{L^2(\eps,T)} \\
 &\quad + \int_\eps^T \Lambda_1 [\phi_1(p_1(t)-a) - \phi_1(p_2(t)-a) + \phi_1(p_2(t)-a) - \phi_2(p_2(t)-a)]\;dt\\
 &\le C_8 \|p_1(t)-p_2(t)\|_{L^\infty(\eps,T)} + C_9\|\Phi_1 - \Phi_2\|_{L^2((0,1))} \le C_{10} \delta_p
\end{align*}
using that for positive times $t\ge \eps$ (and away from the boundary) $\phi_1$ is Lipschitz continuous.
Next we have
\begin{align*}
 (III) \le C_{11} \|\Lambda_1 - \Lambda_2\|_{L^2(\eps,T)} \le C_{11} \delta_\lambda.
\end{align*}
and finally
\begin{align*}
 (IV) \le C_{12}\|p_1-p_2\|_{C^1(\eps,T)} \le C_{12}\delta_p.
\end{align*}
Combining all estimates and taking the supremum over all $\Psi \in L^2((0,1))$ with $\|\Psi\|_{L^2((0,1))}=1$, we finally obtain
\begin{align}\label{e:stabf_delta}
 \|f_1(x,T) - f_2(x,T)\|_{L^2((0,1))} \le C_{13} \delta_p + C_{11}\delta_\lambda
\end{align}
Taking $C_6 = \max(C_{13},C_{11})$ yields the assertion.
\end{proof}
\begin{rem} The estimates of Lemma \ref{lem:stab_u} and \ref{lem:stab_initial} show that, for $\alpha > 0$, the reconstruction of the unknown buyer vendor distribution $f(x,T)$ is actually a well-posed problem, at least for sufficiently smooth perturbations of $p$. This is due to the fact that we are solving a regularized optimization problem. The price to pay is that the term involving $f(x,\eps)$ in \eqref{eq:duality_transformed} does not vanish. However, since $f(x,\eps)$ is fixed, is does not appear in our stability estimates.
\end{rem}
For the next result, we choose perturbed prices $p_1$ and $p_2$ such that $\vert p_1(T) - p_2(T) \vert < 2a$ and assume w.l.o.g. that $p_1(T) \le p_2(T)$ and make the following additional assumptions:
\begin{enumerate}[label=(A\arabic*),start=4]
\item $\delta_p < a$, \label{a:deltabound}
\item $f_1(x,T),\,f_2(x,T) \in L^2(-L,L) \cap H^4(I)$ with $I \subset (p_2(T)-a,p_1(T)+a)$\label{a:highreg}
\item $\|f_1(x,T) - f_2(x,T)\|_{H^4(I)} \le C_{6}( \delta_p + \delta_\lambda)$ \label{a:normest}
\end{enumerate}

\begin{rem}\label{rem:regularity}
We mention that indeed it is natural to assume strong regularity of $f$ in a neighbourhood of $p(T)$ for $T > 0$, since it  locally arises as the solution of a heat equation. On the other hand we need to expect some singularities around $p(T)-a$ and $p(T)+a$ due to the singular source terms. Thus (A5) seems completely natural for forwards solutions of the price formation model. Moreover, it can also be verified that $f(\cdot,T)$ reconstructed via \eqref{eq:duality_transformed} has local $H^4$-regularity, which follows from using $\Psi$ supported in $I$ and an analysis of the solution of the parabolic equation for $\Phi$, which can be estimated in terms of the $H^{-4}$ norm of the initial value. 
\end{rem}
In the following we analyse the forward propagation for $t>T$ in a small time interval. We denote the new initial value by $f_{i,0}:=f_i(\cdot,T)$.
First note that using the same localisation strategy as in \cite{MMPW2009} (i.e. multiplying the solution to \eqref{e:ll} with a smooth cut-off function that has support inside the interval $I$), implies

\begin{align}\label{eq:localisation}
\begin{aligned}
\|f_i\|_{L^{\infty}\left((T, T+\gamma ) ; H^{\beta}\left(I_{2}\right)\right)} & \leq C_{14}\left(\|f_i\|_{L^2\left(((T, T+\gamma)) ; H^1((-L,L))\right)}+\left\|f_{i,0}\right\|_{H^{\beta}\left(I\right)}\right) \\ 
& \leq C_{15}\left(\left\|f_{i,0}\right\|_{L^{2}((-L,L))}+\left\|f_{i,0}\right\|_{H^{\beta}\left(I\right)}\right) \text{ for } \beta \le 4,
\end{aligned},
\end{align}
with $\gamma > 0$ to be fixed later on and where $f_i$ is the solution to \eqref{e:ll} with the reconstructed initial datum $f_i(x,T)$ that additionally satisfies \ref{a:deltabound}-\ref{a:normest}. Furthermore, $I_2$ is an interval that is compactly contained in $I$. 
This allows us to derive the following estimates on terms of the form $(\partial_x K_N)\ast f$, where we denote by $K_N(x,t)$ the heat kernel with Neumann boundary conditions on $[-L,L]$, see e.g. \cite[Section 6.4]{Cannon1984}, and furthermore use the notation $$K_N^T(x,t) := K_N(x,t-T).$$ 
\begin{lemma}\label{lem:heat_lq} For given $T>0$, initial values $f_0,\,f_{1,0}$ and $f_{2,0}$ at time $T$ satisfying \ref{a:deltabound}, \ref{a:highreg} we have for $t\in [T,T+\gamma]$ with $\gamma$ sufficiently small 
$$
|((\partial_x K_N^T)\ast (f_{1,0}-f_{2,0}))(p(t),t)|  \le C_{16}\left(\|f_{1,0} - f_{2,0}\|_{L^2((-L,L))} + \|f_{1,0} - f_{2,0}\|_{H^4(I)}\right).
$$
For two continuous functions $p_1(t)$ and $p_2(t)$, we have
$$
|((\partial_x K_N^T)\ast f_0)(p_1(t),t) - ((\partial_x K_N)\ast f_0)(p_2(t),t)| \le C_{17} \|p_1-p_2\|_{C([0,T])}.
$$
\end{lemma}
\begin{proof}
First note that $(\partial_x K_N^T)\ast f$ is the solution to the heat equation with homogeneous Neumann boundary condition, 
zero right hand side and initial datum $f$. Then, the first estimate is a direct consequence of \eqref{eq:localisation} applied to such an solution 
with initial datum $f_{1,0} - f_{2,0}$. The second one follows from the fact that, as for $t$ sufficiently small, $p_1(t)$ and $p_2(t)$ are in $I_2$ and thus, 
using again \eqref{eq:localisation}, the derivative of a solution to the heat equation that appears on the left hand side is Lipschitz continuous.
\end{proof}

We are now in a position to state the stability result for future prices.
\begin{lemma}\label{lem:stab3} Let assumptions \ref{a:plbounds}--\ref{a:normest} be satisfied and denote by $f_1$ and $f_2$ the solution to \eqref{e:ll} on the time interval
$[T,T+\gamma]$ with initial data $f_1(x,T)$ and $f_2(x,T)$, reconstructed from measurements $p_1,\Lambda_1$ and $p_2,\Lambda_2$ in $[0,T]$, respectively. 
Then there exists a constant $\gamma > 0$ and we the corresponding prices $p_1$ and $p_2$ for $t \in (T,T+\gamma)$ satisfy the estimate
$$
 \| p_1( t) - p_2( t)\| \le C_{18} e^{C_{19}(t-T)} (\delta_p + \delta_\Lambda).
 $$
\end{lemma}
Note that unfortunately, we cannot give a lower bound on the quantity $\gamma$ as it depends in a non-linear and non-local fashion on the initial datum via the solution of the equation. However, the proof below shows that as the transaction rate $\Lambda$ increases, also $\gamma$ becomes larger which agrees with the modelling.
\begin{proof} 
Due to assumption \ref{a:highreg} we can invoke \cite{MMPW2009}[Lemma 2.5] to show that for $\gamma$ sufficiently small 
(depending on $f_i(x,T)$, $i=1,2$) the corresponding transaction rates $\Lambda_1,\, \Lambda_2$ are strictly positive 
on $[T,T+\gamma]$. Furthermore, \ref{a:highreg} implies that $p_1(t),\; p_2(t)$ are in $C^1([T,T+\gamma])$. Now Duhamel's formula allows 
us to express the solutions $f_i(x,t)$ to \eqref{e:ll} as 
\begin{align}
f_i(x,t) = K_N^T \ast f_{i,0} + \int_T^t \Lambda_i(\tau) [K_N^T(x-p_i(\tau)+a,t-\tau)-K_N^T(x-p_i(\tau)-a,t-\tau)]\;d\tau
\end{align}
Taking the space derivative and evaluating at $x=p_i(t)$ we obtain
\begin{align}\label{eq:lambda_duhamel}
\begin{aligned}
\Lambda_i(t) &= \partial_x f_i(p(t),t) = (\partial_x K_N^T) \ast f_{i,0}(p(t),t) \\
&+ \int_T^t \Lambda_i(\tau) [\partial_x K_N^T(p_i(t)-p_i(\tau)+a,t-\tau)-\partial_x K_N^T(p_i(t)-p_i(\tau)-a,t-\tau)]\;d\tau.
\end{aligned}
\end{align}
Subtracting \eqref{eq:lambda_duhamel} for $i=1,2$ and using the linearity of the convolution, we obtain
\begin{align}\label{eq:rate_duhamel}
\begin{split}
\Lambda_1(t) - \Lambda_2(t) &= (\partial_x K_N^T) \ast (f_{2,0}(p_2(t),t) - (\partial_x K_N^T  \ast (f_{2,0}(p_1(t),t) \\
& \quad + (\partial_x K_N^T) \ast (f_{2,0}-f_{1,0})(p_1(t),t)\\
& \quad + \int_0^t \left(\Lambda_1(\tau)- \Lambda_2(\tau)\right) \theta_1(t,\tau) + \Lambda_2(\tau)\left(\theta_1(t,\tau) - \theta_2(t,\tau)\right)\;d\tau,
\end{split}
\end{align}
with 
$\theta_i(t,\tau) = [\partial_x K_N^T(p_i(t)-p_i(\tau)+a,t-\tau)-\partial_x K_N(p_i(t)-p_i(\tau)-a,t-\tau)].$
As the $p_i(t)$ are continuous, choosing $\gamma$ sufficiently small guarantees that the derivatives of $K_N^T$ appearing in the definition of $\theta_i$ are always evaluated 
away from their singularity, in particular they are bounded and locally Lipschitz-continuous, which implies with the local Lipschitz constant $\lambda$ 
$$ | \theta_1(t,\tau) -\theta_2(t,\tau) | \leq \lambda \vert p_1(t)-p_1(\tau) - p_2(t) +p_2(\tau) \vert \leq 2 \lambda
\|p_1(t) - p_2(t)\|_{C([T,T+\gamma])}. $$
 Taking the absolute value on both sides of \eqref{eq:rate_duhamel} and using Lemma \ref{lem:heat_lq} implies
\begin{align*}
|\Lambda_1(t) - \Lambda_2(t)| &\le C_{20} \left(\|f_1^0 - f_2^0\|_{L^2((-L,L))}+\|f_1^0 - f_2^0\|_{H^4(I)}\right) + C_{21} \|p_1(t) - p_2(t)\|_{C([T,T+\gamma])} \\
&+ C_{22} \int_T^t |\Lambda_1(\tau)- \Lambda_2(\tau)|\;d\tau,
\end{align*}
so that Gronwall's lemma implies, together with (A3) and (A6), yields
\begin{align}\label{eq:estimate_gronwall_lambda}
|\Lambda_1(t) - \Lambda_2(t)| \le C_{23}(\delta_p + \delta_\Lambda)e^{C_{22}t}.
\end{align}
Next we exploit the fact that $f_i(p_i(t),t)=0$ by taking the time derivative, which gives
\begin{align*}
 0 = \frac{d}{dt} f_i(p_i(t)) = \dot p_i(t) \partial_x f_i(p_i(t),t) + \partial_t f_i(p_i(t),t),\quad i=1,2.
\end{align*}
Subtracting the above equation for $i=1$ and $i=2$ respectively, using the definition of $\Lambda_i$ and integrating in time we obtain, for $T \le t \le T+\gamma$
\begin{align}
\begin{split}
p_1( t) - p_2( t) =& \int_T^{ t}\left( \frac{\Lambda_2(s)\partial_t f_1(p_1(s),s)-\Lambda_1(s)\partial_t f_2(p_2(s),s)}{\Lambda_1(s)\Lambda_2(s)}\right)\;ds \\
 & + (p_1(T) - p_2(T))
 \end{split}
\end{align}
Denoting by $\underline{\Lambda} = \inf_{T \le s \le T+\gamma} \Lambda_1(s)\Lambda_2(s)$ and using (A3) this yields
\begin{align*}
 |p_1( t) - p_2( t)| &\le \frac{1}{\underline \Lambda } \int_T^{ t} |(\Lambda_2(s)-\Lambda_1(s))\partial_t f_1(p_1(s),s)| \;ds\\
 &+\frac{1}{\underline \Lambda } \int_T^{ t} |\Lambda_1(s)(\partial_t f_1(p_1(s),s) - \partial_t f_2(p_1(s),s))|\;ds \\
 &+\frac{1}{\underline \Lambda } \int_T^{ t} |\Lambda_1(s)(\partial_t f_2(p_1(s),s) - \partial_t f_2(p_2(s),s))|\;ds + \delta_p.
\end{align*} 
As a consequence of \eqref{eq:localisation}, $\partial_t f_i(\cdot, t)$ is bounded and Lipschitz continuous. 
Thus using \eqref{eq:estimate_gronwall_lambda}, \ref{a:plbounds} and once more \eqref{eq:localisation} applied to $\partial_t f_1(p_1(t),t) - \partial_t f_2(p_1(t),t)$ (and together with \ref{a:normest}) finally yields the assertion.
\end{proof}
\section{Numerical Simulation}\label{sec:numerics}
We conclude by illustrating the proposed methodologies and confirming the
obtained analytic results with various computational experiments. All simulations are
performed on
the domain $[-L,L]$, which is split into $N$ intervals of length $h$. The
discrete grid points are denoted by $x_i = i h$. We compute solutions at
discrete times $t^k = k\Delta t$, where $\Delta t$ is the discrete time
step. However we will omit all full time-discrete expressions in the
following, to enhance readability.\\

The reconstruction of the
buyer-vendor distribution is based on piecewise linear basis functions. 
Let $V_h$ denote the space of piecewise linear basis functions $\phi_j$, which
satisfy $\phi_j(x_i) = \delta_{ij}$. We wish to reconstruct $\hat{f} \in
V_h$, which is given by
 $\hat{f}(x,T) = \sum_{j=1}^J \hat{f}_j \phi_j(x)$
using the duality estimates \eqref{e:duality}.\\

\textit{Data generation:} We solve the transformed LL model \eqref{eq:heat}
for a given initial buyer-vendor distribution $f_0$. In doing so we transform the
initial distribution $f_0$ via \eqref{e:trans}, and compute the solution to the heat
equation \eqref{eq:heat} using an implicit in time discretization. The
returned discrete price $p^{(k)} = p(t^k)$ corresponds to the
zero levelset of the buyer-vendor distribution $F^{(k)} = F(t^k)$ (computed via linear
interpolation). Note that we use a
finer spatial and temporal discretization to generate the data than in the subsequent
reconstruction. \\

\textit{Steepest descent:} We solve \eqref{eq:minreg1} and the corresponding problem on
$\Omega_\triangleright$ and $\Omega_\triangleleft$ using steepest descent. In doing so, we compute the variational
derivatives of \eqref{eq:lagrange} and obtain the first order optimality
system \eqref{eq:heat_opt} as well as the updates for the controls $u_1$ and
$u_2$. The detailed steps are outlined in the While-Loop of
Algorithm \ref{a:recon}. Here the parameter $\beta^{(l)}>0$ is the step size of the
steepest descent update. We use the Armijo-Goldstein condition
to adjust $\beta^{(l)}$ in the search direction $\mathbf{p}$. We recall that the Armijo-Goldstein condition for a general
functional $\mathcal{J}$ is given by
\begin{align}\label{eq:goldstein}
\mathcal{J}(\mathbf{x} + \beta \mathbf{p})  \leq \mathcal{J}(\mathbf{x}) + \beta \gamma \nabla \mathcal{J}(\mathbf{x})^T \mathbf{p}.
\end{align}
where $\gamma \in (0,1]$. The starting value is set to $\beta^{(0)} = 0.25$, which is then
reduced (up to a maximum of four times) by a factor of $\frac{1}{2}$ until condition \eqref{eq:goldstein} is satisfied.
Note that we transform the computational domains $\Omega_\triangleleft$ and $\Omega_\triangleright$ to $[0,1]$ as 
discussed in Section \ref{s:evolspace} in all simulations. We solve the
forward as well as the adjoint equations using an implicit in time
discretization and piecewise linear basis functions in space.

{\SetAlgoNoLine
  \begin{algorithm2e} 
    \DontPrintSemicolon
    \kwGiven{Price $p_i$ and transaction rate $\lambda_i$ at discrete times  $t_i = i \Delta t$}
    \For{$i = 1\ldots J$}{
      \If{$x_j < p(T)$ where $\phi_i(x_j) = \delta_{ij}$ }{$\psi_1(x) = \phi_i(x)$}
      \Else{$\psi_2(x) = \phi_i(x)$}
      k = 0\;
      \While{$k < $ max. iterations and convergence criterion is not satisfied}{    
        Adjoint equ.: Given $u_1^{i,k}(t)$, $u_2^{i,k}(t)$, $\psi_1(x)$ and $\psi_2(x)$ solve  
        \begin{align*}
          &\partial_t \Phi_1^{i,k}(x,t) + \partial_{xx} \Phi_1^{i,k}(x,t) = 0 &&\partial_t \Phi_2^{i,k}(x,t) + \partial_{xx} \Phi_2^{i,k}(x,t) = 0 \\
          &\partial_x \Phi_1^{i,k}(- L,t ) = 0  &&\partial_x \Phi_2^{i,k}(L,t ) = 0\\
          &\Phi_1^{i,k}(p(t),t) = u^{i,k}_1(t)  &&\Phi_2^{i,k}(p(t),t) = u^{i,k}_2(t)\\
          &\Phi_1^{i,k}(x,T) = \psi_1(x)  &&\Phi_2^{i,k}(x,T) = \psi_2(x)
        \end{align*}
        \indent Forward equ.: For $G_1^{i,k}(x,0) = -\Phi_1^{i,k}(x,0)$ and $G_2^{i,k}(x,t) = -\Phi_2^{i,k}(x,0)$ solve
        \begin{align*}
          -&\partial_t G_1^{i,k}(x,t) + \partial_{xx} G_1^{i,k}(x,t)= 0  &&-\partial_t G_2^{i,k}(x,t) + \partial_{xx} G_2^{i,k}(x,t)= 0\\
          &\partial_x G_1^{i,k}(-L,t) = 0 &&\phantom{--}\partial_x G_2^{i,k}(L,t) = 0 \\
          & G_1^{i,k}(p(t),t) = 0 && \phantom{--}G_2^{i,k}(p(t),t) = 0  
        \end{align*}
        \indent    Update controls $u_1^{i,k} = u_1(t)$ and $u_2^{i,k} = u_2(t)$ using a step size $\beta^{(l)}$, which satisfies \eqref{eq:goldstein}:
        \begin{align*}
          &u_1^{i,k+1}(t) = u_1^{i,k}(t) - \beta^{(l)} \left(\alpha u_1^{i,k}(t) + \partial_x G_1^{i,k}(p(t),t)\right)\\
          &u_2^{i,k+1}(t) = u_2^{i,k}(t)  - \beta^{(l)} \left(\alpha u_2^{i,k}(t) - \partial_x G_2^{i,k}(p(t),t)\right).
        \end{align*}
        k = k+1\;
      }
    }
    Reconstruct solution $\hat{f}(x) = \sum_j \hat{f}_j \phi_j(x) $:
    \begin{align*}
      &\sum_j\int_{-L}^{p(T)} \hat{f}_j \phi_j(x) \phi_i(x) dx = \int_{-L}^{p(\eps)} f(x,0) \Phi_1^i(x,0) dx + \int_\eps^T \Lambda(t)(\Phi_1^i(p(t)-a) - u^i_1(t)) dt\\
      &\sum_j \int_{p(T)}^L \hat{f}_j \phi_j(x) \phi_i(x) dx = \int_{p(\eps)}^{L} f(x,0) \Phi_2^i(x,0) dx + \int_\eps^T \Lambda(t)(u^i_2(t) - \Phi_2^i(p(t)-a)) dt
    \end{align*}
    \caption{Reconstruction of $\hat{f}(x,T)$.}
    \label{a:recon}
  \end{algorithm2e}
}

\subsubsection*{Identifiabilty for different initial conditions.} In the first
experiment we set $L = 0.5$ and the final time to $T=0.25$. We
split the spatial domain $[-0.5,0.5]$ into $200$ elements and the time
interval $[0,0.25]$ into $125$ time steps. The initial datum is set to
\begin{align}\label{eq:init1}
  f_0(x) = (x+0.75)(x-0.65)(x-0.05).
\end{align}
We approximate the final buyer-vendor distribution using $J=50$ basis
functions. Furthermore we choose the following  parameters
\begin{align*}
\alpha = 0.1, ~\beta^{(0)}=0.25, \gamma = 0.2, \textrm{max. iterations} = 250 \textrm{ and max
  error} = 10^{-5}. 
\end{align*}  
Figure \ref{f:mon_price} shows the reconstructed and computed function $F$ (the
latter computed by solving the heat equation \eqref{eq:heat} with the
transformed initial datum $F_0$). We observe a good agreement, with small
artefacts at the boundary and the buyer-vendor interface. The corresponding controls are shown in Figure \ref{f:mon_controls}.
\begin{figure}
\centering
 \subfigure[Price dynamics]{\includegraphics[width=0.45\textwidth]{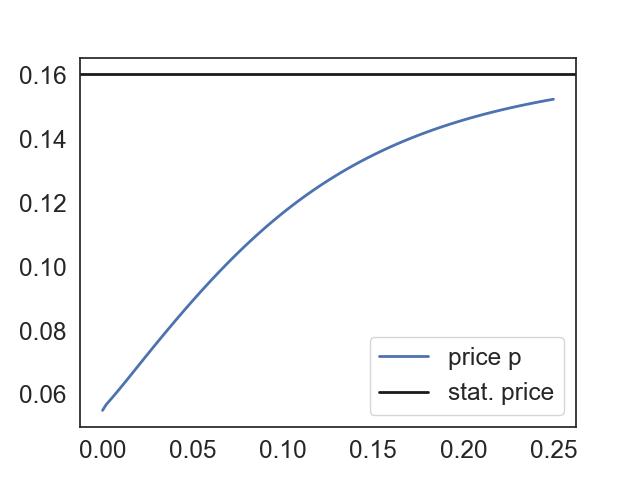}}
 \subfigure[Reconstruction]{\includegraphics[width=0.45\textwidth]{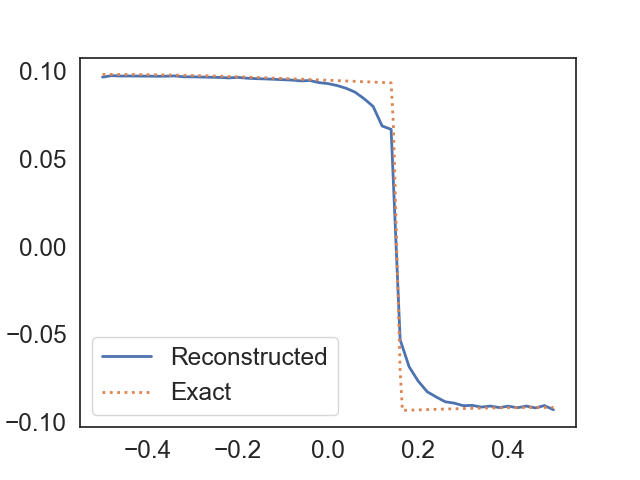}}
 \caption{Left: Evolution of the price $p=p(t)$; Right: Reconstructed and computed buyer-vendor distributions $F$.}
 \label{f:mon_price}
\end{figure}

\begin{figure}
\centering	
 \subfigure[Controls $u_1$]{\includegraphics[width=0.425\textwidth]{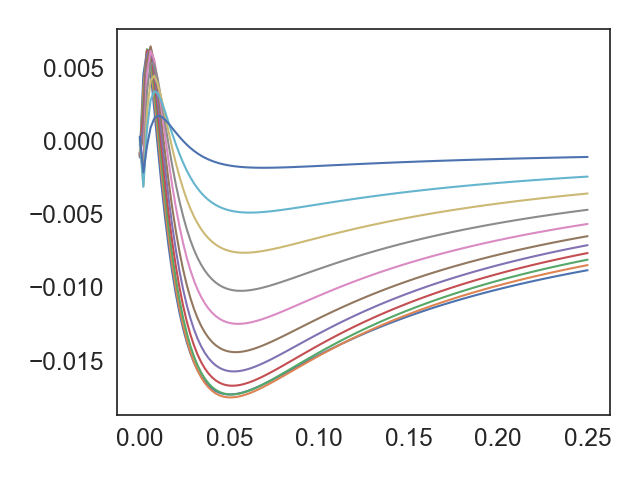}}
 \subfigure[Controls $u_2$]{\includegraphics[width=0.425\textwidth]{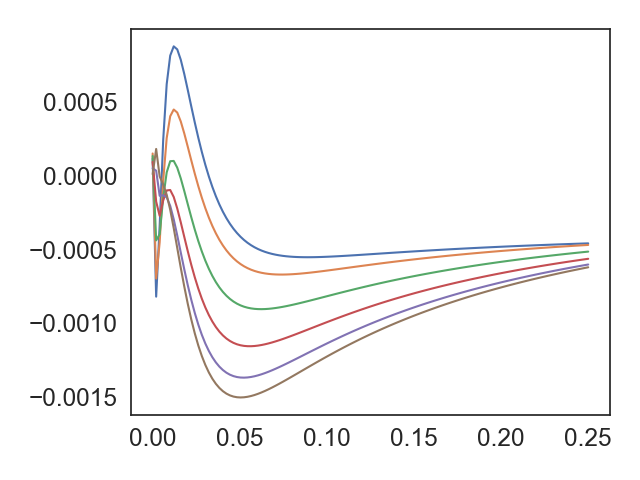}}
 \caption{Evolution of the controls $u_1$ and $u_2$ (plotted at every $15^{\text{th}}$ timestep).}
 \label{f:mon_controls}
\end{figure}

\noindent Next we choose a slightly different initial datum, in particular
\begin{align*}
 f_0(x) = (x+0.75)(x-0.65)(x+0.05)
\end{align*}
In this case the price is not monotone, see Figure \ref{f:non_mon_price}. However, the quality of the reconstructions is comparable to the one of
the previous example.
\begin{figure}
\centering	
 \subfigure[Price dynamics]{\includegraphics[width=0.45\textwidth]{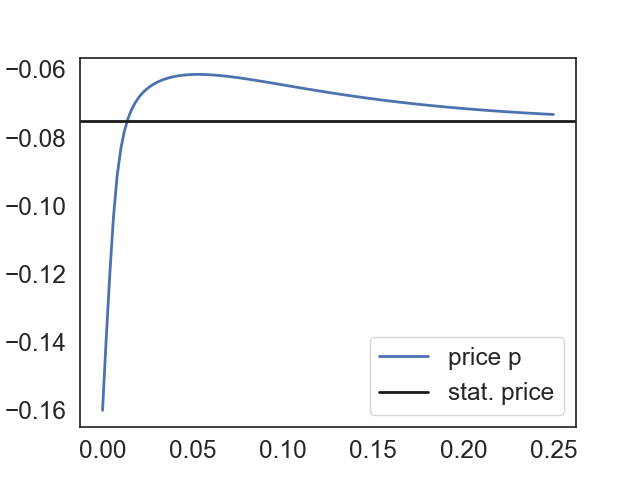}}
 \subfigure[Reconstruction]{\includegraphics[width=0.45\textwidth]{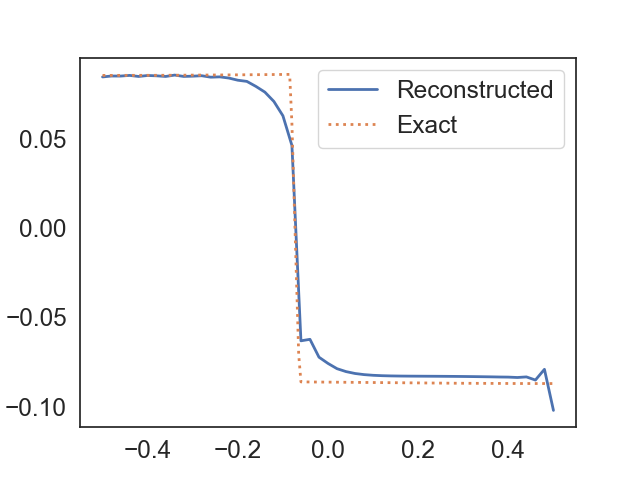}}
  \caption{Left: Evolution of the price $p=p(t)$; Right: Reconstructed and computed buyer-vendor distributions $F$.}
 \label{f:non_mon_price}
\end{figure}

\subsubsection*{Stability of $\hat{f}(x,T)$} 
Next we are interested in the stability of the reconstruction $\hat{f}(x,T)$ with respect to perturbations in the price. Lemma \ref{lem:stab_initial} and in particular estimate \eqref{e:stabf_delta}
state that the difference in the reconstructions is bounded by the difference in the prices and transaction rates. We consider the following perturbation of the unperturbed price 
$p_{\delta,0}$:
\begin{align*}
 p_{\delta,k} = p_{\delta,0} + k \delta \sin(\pi t) \text{ for } k=1, \ldots, K.
\end{align*}
Note that the perturbed price is still in $C^1$ and that $p_{\delta,0}(0) = p_{\delta,k}(0)$ for all $k=1, \ldots K$. 
We start with the same the initial datum as in the first example, that is \eqref{eq:init1}, and set $\delta = 0.01$ and $K=13$. 
All other parameters are the same as in the first example, except that we reconstruct the final profile using $J = 80$ basis functions.
Figure \ref{f:diffdelta} illustrates the linear increase of the error in the controls and the reconstruction as the noise level increases. For larger values of $\delta$ this agrees with the theoretical results of Lemma \ref{lem:stab_u} and Lemma \ref{lem:stab_initial}, respectively. For small $\delta$ a saturation effect due to the numerical discretization error occurs.

\begin{figure}
 \centering
 \subfigure[Difference in the $L^2$ norm of the controls]{\includegraphics[width=0.5\textwidth]{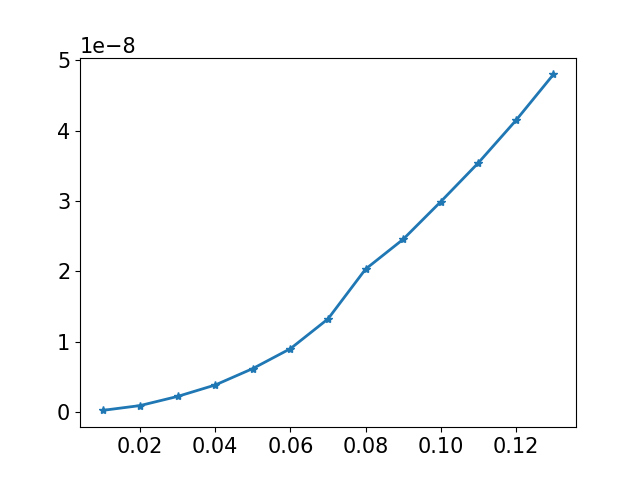}}
 \subfigure[Difference in the $C^1$ norm of the reconstructions]{\includegraphics[width=0.45\textwidth]{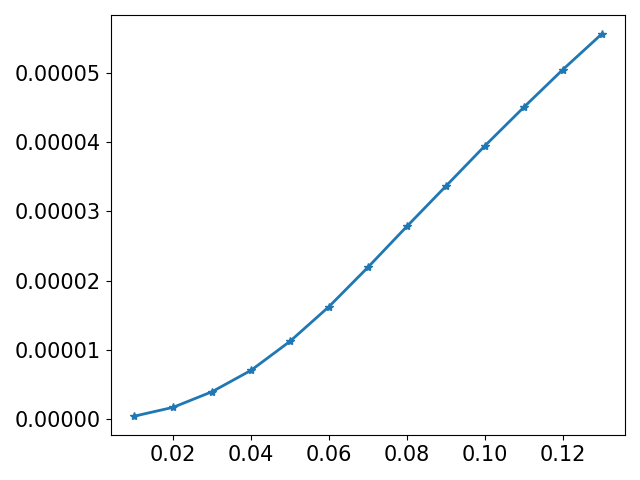}}
 \caption{Difference in the controls and the reconstructions for different values of $\delta$.}\label{f:diffdelta}
\end{figure}

\subsubsection*{Predicting price dynamics}
We conclude with an example where we use the reconstructed buyer-vendor distribution $f(x,T)$ to estimate future price dynamics and illustrate the influence of noise in those.
We consider an initial datum of the form
\begin{align*}
 f_0(x)= \sin(2 \pi x),
\end{align*}
on the domain $\Omega = [-0.5,0.5]$, that is $L = 0.5$. The computational domain is split into $100$ intervals, the time horizon $[0,0.5]$ into $100$ time steps. Since the number
of buyers equals the number of vendors the exact price is constant and equal to zero. The computed price curve, that is the dark blue line in Figure \ref{f:pred_price}, converges 
quickly to the true equilibrium price (being slightly off at the beginning due to numerical errors caused by the nonlinear transformation of the initial datum).\\
We then use the original price curve on the time interval $[0,0.5]$, denoted by $p_{\delta,0}$, as well as the perturbed prices
\begin{align*}
 p_{\delta, k} = p_{\delta, 0} + \delta_0 \sin(4 k \pi t)
\end{align*}
to reconstruct the solution at time $t=0.5$, that is $f(x,T=0.5)$. Note that the sinusoidal perturbations ensure that the final distribution of buyers and vendors equals the initial
distribution. These reconstructed buyer-vendor profiles are then used to compute the future price dynamics. We observe a jump in the price in all prices at time $t=0.5$, which 
is caused by the numerical error made in the reconstruction. However, the price then converges quickly to an equilibrium value, which is close to the true price (dark blue line).\\

\begin{figure}
 \centering
 \includegraphics[width=0.6\textwidth]{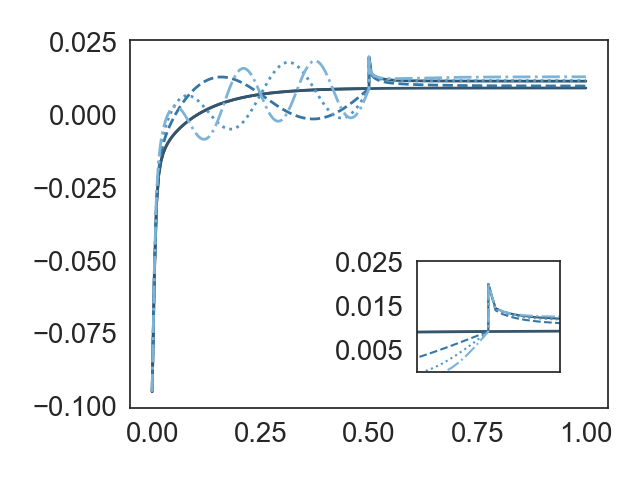}
 \caption{Influence of the perturbations on future price dynamics. The dark blue line (solid) shows the unperturbed evolution of the price up to time $t=1$. The lighter blue lines (dashed) correspond to different perturbations up to time $t=0.5$. A time $t=0.5$, the buyer vendor distribution is reconstructed and then used to calculate the future price dynamics. The inset shows the jump in the price caused by initialising the simulation using the reconstructed buyer-vendor distribution at time $t=0.5$}\label{f:pred_price}
\end{figure}

\noindent The parameters used in the reconstruction Algorithm \ref{a:recon} are set to
\begin{align*}
 \alpha = 0.05, ~\gamma = 0.1, \textrm{ max iterations } 250,\textrm{ and max error } 10^{-5}. 
\end{align*}
The final buyer vendor distribution $f(x,T=0.5)$ was computed using $J=80$ basis functions. 

\section{Summary and Outlook}
We studied a data assimilation problem for a parabolic nonlinear free boundary problem. This partial differential equation describes the 
evolution of the price, that is the free boundary, in a large economic market. We developed an analytical and computational framework for 
the corresponding data assimilation problem,  which is based on a previous work by Puel et al., see \cite{Puel2002}. The free boundary splits the original problem into two 
parts, each of them defining a separate optimal control problem. We discussed analytic properties of the respective problems and derived
stability estimates for the controls and reconstructed unknown buyer-vendor distribution in the presence of noise. Finally we confirmed and illustrated our
results with computational experiments.\\
We believe that the developed framework provides the basis for more general data assimilation problems in price formation. In \cite{BCMW2013} Burger et al. considered a Boltzmann type price formation
model, which allows for more complex trading mechanisms.  This problem is a system of nonlocal reaction-diffusion equations on the whole domain, where multiple prices (even with continuous distribution) and transaction rates can appear. Analogous questions can be asked for this problem if only the expectation of the price is to be predicted, but the problem could also be extended to a stochastic distribution of the price.

\section*{Acknowledgements}
MB acknowledges support by ERC via Grant EU FP 7 - ERC Consolidator Grant
615216 LifeInverse. MTW acknowledges financial support from the Austrian
Academy of Sciences \"OAW via the New Frontiers Grant NST-001 and the EPSRC
via the First Grant EP/P01240X/1.

\bibliographystyle{plain}
\bibliography{priceinv}

\end{document}